\documentclass{article}

\usepackage{amsmath}
\usepackage{amsfonts}
\usepackage{amssymb}
\usepackage{enumerate} 

\newtheorem{theorem}{Theorem}

\newtheorem{lemma}[theorem]{Lemma}

\newenvironment{proof}[1][Proof]{\textbf{#1.} }

\begin{document}

\title{On Mixed Concatenations of Fibonacci and Lucas Numbers Which are Fibonacci Numbers}
\author{Alaa ALTASSAN$^{1}$ and Murat ALAN$^{2}$ \\
$^{1}$King Abdulaziz University, Department of Mathematics,\\
P.O. Box 80203, Jeddah 21589, Saudi Arabia\\
e-mail: aaltassan@kau.edu.sa\\
$^{2}$Yildiz Technical University\\
Mathematics Department, 34210, Istanbul, Turkey.\\
e-mail: alan@yildiz.edu.tr
}

\maketitle

\begin{abstract}
Let $(F_n)_{n\geq 0}$ and $(L_n)_{n\geq 0}$ be the Fibonacci and Lucas sequences, respectively. In this paper we determine all Fibonacci numbers which are mixed concatenations of a Fibonacci and a Lucas numbers. By mixed concatenations of $ a $ and $ b $, we mean the both concatenations $\overline{ab}$ and $\overline{ba}$ together, where $ a $ and $ b $ are any two non negative integers. So, the mathematical formulation of this problem leads us searching the solutions of two Diophantine equations $ F_n=10^d F_m +L_k $ and $ F_n=10^d L_m+F_k $ in non-negative integers $ (n,m,k) ,$ where $ d $ denotes the number of digits of $ L_k $ and $ F_k $, respectively. We use lower bounds for linear forms in logarithms and reduction method in Diophantine approximation to get the results.
\end{abstract}

\section{Introduction}
Let $(F_n)_{n\geq 0}$ and $(L_n)_{n\geq 0}$ be the Fibonacci and Lucas sequences given by $F_0=0$, $F_1=1$, $L_0=2$, $L_1=1$, $F_{n+2}=F_{n+1}+F_{n}$ and $L_{n+2}=L_{n+1}+L_{n}$ for $n\geq 0,$ respectively. 

In recent years, the numbers in Fibonacci, Lucas or some similar sequences which are concatenations of two or more  repdigits  are investigated by a series of papers \cite{Alahmadi2,Dam1,Dam2, EK2,Qu,Trojovsky}. In the case of the concatenation of binary recurrent sequences, there is a general result due to Banks and Luca \cite{Banks}. In \cite{Banks}, they proved that if $ u_n $ is any binary recurrent sequence of integers, then only finitely many terms of the sequence $ u_n $ can be written as concatenations of two or more  terms of the same sequence $ u_n $ under some mild hypotheses on $ u_n.$ In particular, they proved that 13, 21 and 55 are the only Fibonacci numbers which are non trivial concatenations of two terms of Fibonacci numbers. In \cite{Alan}, Fibonacci and Lucas numbers which can be written as concatenations of two terms of the other sequences also investigated and it is shown that 13, 21 and 34 (1, 2, 3, 11, 18 and 521 ) are the only Fibonacci (Lucas) numbers which are concatenations of two Lucas (Fibonacci) numbers.
 
In this paper, we study the mixed concatenations of these famous two sequences which forms a Fibonacci number. By mixed concatenations of $ a $ and $ b $, for any two non negative integers $ a $ and $ b ,$ we mean the both concatenations $\overline{ab} $ and $ \overline{ba},$ together. So, we search all Fibonacci numbers of the form $\overline{F_mL_k}$, that is the concatenations of $ F_m $ and $ L_k $, as well as of the form $\overline{L_mF_k},$ that is the concatenations of $L_m $ and $ F_k.$ In other words, by the mathematical expression of this problem, we solve the Diophantine equations
\begin{equation}
F_n=10^d F_m +L_k
\label{FcFL}
\end{equation}
and 
\begin{equation}
F_n=10^d L_m+F_k
\label{FcLF}
\end{equation}
in non negative integers $ (n,m,k) $ where $ d $ denotes the number of digits of $ L_k $ and $ F_k $, respectively and we get the following  results.

\begin{theorem}
\label{main1}
All Fibonacci numbers which are concatenations of a Fibonacci and a Lucas numbers are only the numbers 1, 2, 3, 13, 21 and 34.
\end{theorem}

\begin{theorem}
\label{main2}
All Fibonacci numbers which are concatenations of a Lucas and a Fibonacci numbers are only the numbers 13 and 21.
\end{theorem}

In the next chapter, we give some details of the methods we used in this study to prove the above theorems. In fact, we mainly use two powerful tools. The first one is the theory of non zero linear forms in logarithms of algebraic numbers which is due to Matveev \cite{Matveev} and the second one is reduction method based on the theory of continued fractions given in \cite{DP} which is a version of Baker-Davenport lemma \cite{Baker-Davenport}. In the third section, we give the proofs of the above theorem. All calculations and computations are made with the help of the software \textsf{Maple}.

\section{Preliminaries}
Let $F_n$ and $L_n$ be the Fibonacci and Lucas numbers, respectively. The Binet formula for Fibonacci and Lucas numbers are given by
$$
F_n = \dfrac{\alpha^n-\beta^n}{\sqrt{5}}, \qquad   L_n=\alpha^n+\beta^n \quad  \quad n\geq 0, 
$$
where
\[
\alpha=\frac{1+\sqrt{5}}{2} \quad \text{and} \quad \beta=\frac{1-\sqrt{5}}{2}
\]
are the roots of the equation $x^2-x-1=0.$ By using the Binet formula of these sequences, one can see that, by induction,
\begin{equation}
\label{F1}
\alpha^{n-2} \leq F_n \leq \alpha^{n-1} 
\end{equation}
and 
\begin{equation} 
\label{L1}
\alpha^{n-1}\leq L_n \leq 2\alpha^{n} 
\end{equation}
hold for all $n\geq 1$ and  $n\geq 0,$ respectively.

Let $\eta$ be an algebraic number of degree $d$ with minimal polynomial 
\[
a_0x^d+a_1x^{d-1}+\cdots+a_d=a_0\prod_{i=1}^{d}(x-\eta^{(i)}),
\]
where the $a_i$'s are relatively prime integers with $a_0>0$ and the $\eta^{(i)}$'s are conjugates of $\eta$. Recall that, the logarithmic height of $\eta$ is defined by
\[
h(\eta)=\frac{1}{d}\left(\log a_0+\sum_{i=1}^{d}\log\left(\max\{|\eta^{(i)}|,1\}\right)\right).
\]
In particular, for a rational number $p/q$ with $\gcd(p,q)=1$ and $q>0$,  $h(p/q)=\log \max \{|p|,q\}$. The logarithmic height $ h(\eta) $ has the following properties:
\begin{itemize}
\item[$ \bullet $] $h(\eta\pm\gamma)\leq h(\eta) + h(\gamma)+\log 2$.
\item[$ \bullet $] $h(\eta\gamma^{\pm 1})\leq h(\eta)+h(\gamma)$.
\item[$ \bullet $] $h(\eta^{s})=|s|h(\eta),$ $ s \in \mathbb{Z} $.
\end{itemize}
\begin{theorem}[Matveev's Theorem]
\label{Matveev}
Assume that $\gamma_1, \ldots, \gamma_t$ are positive real algebraic numbers in a real algebraic number field $\mathbb{K}$ of degree $ d_\mathbb{K} $, $b_1,\ldots,b_t$ are rational integers, and 
\[
\Lambda:=\eta_1^{b_1} \cdots \eta_t^{b_t}-1,
\]
is not zero. Then
\[
|\Lambda|>\exp\left(-1.4\cdot 30^{t+3}\cdot t^{4.5}\cdot d_\mathbb{K}^2(1+\log d_\mathbb{K})(1+\log B)A_1\cdots A_t\right),
\]
where $B\geq \max\{|b_1|,\ldots,|b_t|\},$ and  $A_i\geq \max\{d_\mathbb{K}h(\eta_i),|\log \eta_i|, 0.16\},$ for all $i=1,\ldots,t.$
\end{theorem}
We cite the following lemma from \cite{DP}, which is a version of the reduction method based on the Baker-Davenport lemma \cite{Baker-Davenport}, and we use it to reduce some upper bounds on the variables. Recall that, for a real number $ \theta, $ we put $ ||\theta||=\min\{ |\theta -n | :  n \in\mathbb{N} \}, $ the distance from $ \theta $ to the nearest integer.
\begin{lemma} \label{reduction}
Let $M$ be a positive integer, $p/q$ be a convergent of the continued fraction of the irrational $\gamma$ such that $q>6M$, and let $A,B,\mu$ be some real numbers with $A>0$ and $B>1$. If $\epsilon:=||\mu q||-M||\gamma q|| >0$, then there is no solution to the inequality
\[
0< | u\gamma-v+\mu | <AB^{-w},
\]
in positive integers $u,v$ and $w$ with
\[
u\leq M \quad\text{and}\quad w\geq \frac{\log(Aq/\epsilon)}{\log B}.
\]
\end{lemma}

\section{Proof of Theorem \ref{main1} and \ref{main2}}
\textbf{Proof of Theorem \ref{main1}:}
Assume that the equation \eqref{FcFL} holds. We will need the relations among the variables $ n, m, k $ and $ d $ through this section. Note that, we may write the number of digits of $ L_k $ as $ d=\lfloor \log_{10}L_k \rfloor +1$  where $ \lfloor \theta\rfloor $ is the floor function of $ \theta $ , that is the greatest integer less than or equal to $ \theta .$ 
Thus,
\begin{align*}  
d=\lfloor \log_{10}L_k \rfloor +1 \leq  1+\log_{10}L_k  & \leq 1+ \log_{10}{ (2\alpha^{k} )}= 1+ {k}\log_{10}{ \alpha } +\log_{10}{ 2 }  \\
 \end{align*}
and
$$  d=\lfloor \log_{10}L_k \rfloor +1 > \log_{10}L_k  \geq  \log_{10}{ \alpha^{k-1} } \geq {(k-1)} \log_{10}{ \alpha } .$$
From the above relations we may get more explicit bounds for $ d $ as
\begin{equation}
\label{d}
\frac{k-1}{5} < d < \frac{k+6}{4},
\end{equation}
by using the fact that $ (1/5)< \log_{10}{ \alpha } \cong 0.208... <(1/4) $ and $ \log_{10}{ 2 } < 0.31 .$ 

In particular,
$$ L_k = 10^{ \log_{10}L_k } <10^d < 10^{1+ \log_{10}L_k } <10 L_k .$$

From the last inequality together with \eqref{FcFL} we write
$$\alpha^{n-2} \leq F_n=10^d F_m +L_k \leq 10 L_k F_m +L_k <11 F_mL_k < 22 \alpha^{m+k-1} \leq \alpha^{m+k+6} $$
and
$$ \alpha^{n-1} \geq F_n=F_m 10^d +L_k > F_m L_k+L_k >  F_m L_k \geq \alpha^{m+k-3}.$$
Hence, we have that
\begin{equation}
\label{n}
m+k-2 < n < m+k+8.
\end{equation}
Before further calculations, we wrote a short computer program to search the variables $ n,m $ and $ k $ satisfying \eqref{FcFL} in the range $ 0\leq m,k <200 $ and we found only the Fibonacci numbers given in Theorem \ref{main1}. So from now on we may assume that $ \max \{m,k\} \geq 200. $

Note that, we may take $ n-k \geq 4 .$ Indeed, using the well-known fact $ L_k=F_{k+1}+F_{k-1} $, see for example \cite{Koshy}, we may write equation \eqref{FcFL} as
$$ F_n= 10^d F_m+F_{k+1}+F_{k-1} .$$
Then, clearly $ n \neq k $ and $ n \neq k+1 .$ If $ m=0 ,$ then the case $  F_n=L_k $ is possible only for $ L_k \in \{ 1,2,3 \}, $ that is $ \max \{m,k\} < 3, $ a contradiction. So $ m \neq 0 $ and hence from the inequality
$$  F_n=10^d F_m  +L_k  \geq 10^d +L_k > 2L_k = 2 (F_{k+1}+F_{k-1}) ,$$
we see that the cases $ n = k+2 $ and $ n = k+3 $ are not also possible. So we get that
\begin{equation}
\label{nk4}
n-k \geq 4.
\end{equation}
Using Binet formula for Fibonacci and Lucas sequences, we rewrite equation \eqref{FcFL} as
$$ \dfrac{\alpha^{n} -\beta^{n}}{\sqrt{5}} = \dfrac{\alpha^{m}}{\sqrt{5}} 10^d -  \dfrac{\beta^{m}}{\sqrt{5}} 10^d +L_k,$$
$$  \dfrac{\alpha^{n}}{\sqrt{5}} - \dfrac{ 10^d  \alpha^{m}}{\sqrt{5}} = \dfrac{\beta^{n}}{\sqrt{5}} -  \dfrac{ 10^d  \beta^{m}}{\sqrt{5}} +  L_k. $$
Multiplying both sides of the above equation by $ {\sqrt{5}}/{\alpha^{n}} ,$ and taking absolute values of the both sides of it, we get that
$$ 
\left| 1-  \dfrac{10^{d} }{ \alpha^{n-m}}  \right|  \leq  \dfrac{ |\beta^{n}| }{\alpha^n} +  \dfrac{ 10^d  |\beta^{m}|}{ \alpha^n } + \dfrac{L_k \sqrt{5} }{\alpha^n}
$$
$$                                                                       
 \qquad \qquad  \qquad  \leq  \dfrac{ 1}{\alpha^{2n}} +  \dfrac{ 10 L_k  }{ \alpha^{n+m} } + \dfrac{2 \alpha^k \sqrt{5} }{\alpha^n}  
$$
$$
 \qquad \qquad  \qquad < \dfrac{ 1}{\alpha^{2n}} +  \dfrac{ 20 \alpha^k  }{ \alpha^{n+m} } + \dfrac{2 \sqrt{5} }{\alpha^{n-k}}. 
$$

Since $ 20 < \alpha^{7} ,$ we find that
\begin{equation}
\label{1ineq}  
\Lambda_1 := \left| 1-  \dfrac{10^{d} }{ \alpha^{n-m}}  \right| < \dfrac{ 3}{ \alpha^{n-k-7}}.
\end{equation}
Let $(\eta_1, b_1)=(10, d)$ and  $(\eta_2, b_2)=(\alpha, -(n-m)),$ where $ \eta_1, \eta_2 \in \mathbb{K}=\mathbb{Q}(\sqrt{5}).$ We take $ d_\mathbb{K}=2 $, to be the degree of the real number field $ \mathbb{K}.$ 
Since $  h(\eta_1)=\log{10} ,$  $h(\eta_2)=\dfrac{1}{2} \log {\alpha} ,$ we take $ A_1 = 2\log{10} ,$  $A_2= \log {\alpha}.$ Suppose that $ n-m < d.$ Then from the two relations \eqref{d} and  \eqref{n}, we get that 
$$ k-2<n-m<d<(k+6)/4,$$ 
which implies $ k \leq 4 .$ Since $ L_4=7 ,$  we have $ d=1 .$ Hence $ n=m,$ that is $ F_n=10 F_n+L_k ,$  which is clearly false . So we take
$$
B: = \max\{ |b_i | \} = \max\{ d, n-m\} = n-m.
$$
In \eqref{1ineq},  $ \Lambda_1 \neq 0 .$ Indeed if $ \Lambda_1=0 ,$ then we get that $ \alpha^{n-m}= 10^d \in \mathbb{Q} $, which is possible only for $ n=m $ and we know that this is not the case. So $ \Lambda_1 \neq 0.$ Now we apply Theorem \ref{Matveev} to $ \Lambda_1 $ and we get that
$$
\log{(  \Lambda_1 ) } > -1.4 \cdot 30^5 \cdot 2^{4.5} \cdot 2^2 (1+\log 2)(1+\log {(n-m)}) \cdot 2 \log{10}\cdot \log \alpha .
$$
On the other hand, taking the logarithm of both sides of \eqref{1ineq}, we get also
$$
\log{(  \Lambda_1 ) } < \log 3 - (n-k-7) \log \alpha.
$$
Combining last two inequalities, we get that
\begin{equation}
\label{n-k}
n-k-7 < 2.41 \cdot 10^{10}\cdot (1+\log {(n-m)}).
\end{equation}
We will turn \eqref{n-k} later. Now we rewrite \eqref{FcFL} as
$$ \dfrac{\alpha^{n}}{\sqrt{5}} - \alpha^k - 10^d F_m = \dfrac{\beta^{n}}{\sqrt{5}}+\beta^k ,$$
$$  \dfrac{ \alpha^n}{\sqrt{5}} \left( 1- \alpha^{k-n} \sqrt{5} \right) - 10^d F_m = \dfrac{\beta^{n}}{\sqrt{5}}+\beta^k. $$
Note that, by \eqref{nk4}, $ 1- \alpha^{k-n} \sqrt{5} >0 $  and  $ 1< \dfrac{1}{1 - \alpha^{k-n} \sqrt{5} } <2 $   for  $ n-k \geq 4 .$ So, we may divide, and than take the absolute value of both sides of the last equality to get that
$$ 
\left| 1- \dfrac{ F_m 10^d \sqrt{5} }{  { \alpha^n} \left( 1- \alpha^{k-n} \sqrt{5} \right)   } \right|  < \left(  \dfrac{1}{1 - \alpha^{k-n} \sqrt{5} }    \right) \left(  \dfrac{ |\beta^n | }{\alpha^n}  +  \dfrac{ |\beta^k | \sqrt{5} }{\alpha^n }   \right)
$$
$$
\qquad \qquad <    2   \left(  \dfrac{ 1 }{\alpha^{2n}}  +  \dfrac{  \sqrt{5} }{\alpha^{n+k} }   \right).
$$
So
\begin{equation}
\label{2ineq}
\left| \Lambda_2 \right | :=  \left| 1- \dfrac{ 10^d F_m  \sqrt{5} }{  { \alpha^n} \left( 1- \alpha^{k-n} \sqrt{5} \right)   } \right|  < \dfrac{7}{\alpha^{2k}}.
\end{equation}
Let
$(\eta_1, b_1)=(10, d)$, $(\eta_2, b_2)=(\alpha, -n)$ and $(\eta_3, b_3)= \left( \dfrac{ F_m \sqrt{5} }{ \left( 1- \alpha^{k-n} \sqrt{5} \right)  }  , 1 \right) $.
Again $ \eta_1, \eta_2$ and $ \eta_3 $ are all belongs to the real quadratic number field $ \mathbb{K}=\mathbb{Q}(\sqrt{5}) .$ So we take $ d_\mathbb{K}=2 $, to be the degree of $ \mathbb{K}, $  $h(\eta_1)= \log{10}$, $h(\eta_2)=(1/2) \log(\alpha)$ and for $ h(\eta_3) ,$ we need the properties of logarithmic height given in the Preliminaries, so that
\begin{align*}
h(\eta_3)= h \left( \dfrac{ F_m \sqrt{5} }{  1- \alpha^{k-n} \sqrt{5}   }  \right) & \leq h( F_m ) +h( \sqrt{5}) +h( 1- \alpha^{k-n} \sqrt{5} )    \\
& \leq \log (F_m)+   h( \sqrt{5}) + h(\alpha^{k-n} \sqrt{5} ) +\log 2 \\
& \leq (m-1) \log(\alpha) + \dfrac{ |k-n| }{2} \log \alpha + 1. \\
\end{align*}
Since $ m-1<n-k+1 $ from \eqref{n}, we write 
$ h(\eta_3)< 1+\dfrac{3(n-k)+2}{2} \log(\alpha).$

So we take
$$
A_1 = 2\log {10} ,\quad  A_2= \log \alpha \quad \text{and} \quad A_3= 2+ {(3(n-k)+2)} \log(\alpha).$$
By \eqref{d}, if $ n<d<(k+6)/4 ,$ then we find that $ 4n-6<k $ and hence $ m+3n<8  $ from \eqref{n}, a contradiction since $ \max\{m, k \} \geq 200. $ So
$$ B:=\max\{ |b_i | \} = \max\{ d, n, 1 \}= n.$$
We show that $ \Lambda_2 \neq 0. $ For this purpose, assume that $ \Lambda_2 = 0. $ Then we get  that
$$
\alpha^n- \alpha^{k} \sqrt{5} = 10^d  F_m \sqrt{5},
$$
Conjugating this expression in $ \mathbb{K} $ we obtain
$$
\beta^n+ \beta^{k} \sqrt{5} =  - 10^d F_m \sqrt{5}.
$$
From these last two equality, we find that
$$
\label{d2neq2}
10 \sqrt{5} \leq 10^d F_m  \sqrt{5} = | \beta^n+ \beta^{k} \sqrt{5} |  \leq 2 \max \{  \beta^n,  \beta^{k} \sqrt{5} \} \leq 2 \sqrt{5}, 
$$
a contradiction. So we conclude that $ \Lambda_2 \neq 0. $ Thus, we apply Theorem \ref{Matveev}  to $ \Lambda_2 $ given in \eqref{2ineq} and we get that
\begin{equation}
\label{23}
\log{(  \Lambda_2 ) } > -1.4 \cdot  30^6 \cdot 3^{4.5} \cdot 2^2 (1+\log 2)(1+\log n) 2\log{10} \log \alpha \cdot [ 2+ {(3(n-k)+2)} \log(\alpha)].
\end{equation}
On the other hand, from \eqref{2ineq}, we know that
\begin{equation}
\label{24}
\log{(  \Lambda_2 ) } < \log{7} -2k \log \alpha. 
\end{equation}
Combining \eqref{23} and \eqref{24} we get that
\begin{equation}
\label{2k}
k < 2.24 \cdot 10^{12} \cdot (1+\log n) [2+ (3(n-k)+2) \log(\alpha) ] .
\end{equation}
Now, we focus on two inequalities \eqref{2k} and \eqref{n-k} to get a bound for $ n $ by examining the cases $ k \leq m $ and $ m <k  $ separately. First, assume that $ k \leq m. $ Then $ n-m \leq n-k $ and therefore, from \eqref{n-k}, we find
$$
n-k-7 < 2.41 \cdot 10^{10}\cdot (1+\log {(n-k)})
$$
which means that $ n-k < 8 \cdot 10^{11}.$ So it follows that 
\begin{equation*}
n< 1.7 \cdot 10^{12},
\end{equation*}
since from the left side of \eqref{n}, we know that $ m-2<n-k $ and from the right side of \eqref{n}, $ n<2m+8.$

Let $ m \leq k .$ Then $ n<2k+8. $ From \eqref{n-k}, in particular we have that
\begin{equation}
\label{n-k2}
n-k-7 < 2.41 \cdot 10^{10}\cdot (1+\log {n}).
\end{equation}
Substituting \eqref{n-k2} into \eqref{2k} and using the fact that $ (n-8)/2<k ,$ we find that
\begin{equation}
\label{n2}
n< 7 \cdot 10^{26}.
\end{equation}
So whether $ m \leq k $ or not, in either case, we have that $ n< 7 \cdot 10^{26}. $ Now, we reduce this upper bound on $ n. $
\begin{lemma} \label{mbound}
If the equation \eqref{FcFL} holds, then $ m \leq 150. $  In particular, if $ k \leq m ,$ then the equation \eqref{FcFL} has only solution for $ F_n \in \{ 1, 2, 3, 13, 21, 34 \} .$
\end{lemma}
\begin{proof}
Suppose that $ m>150.$ Let $\Gamma_1 :=d \log {10} -(n-m) \log \alpha.$ Since  $ 141< m-9 < n-k-7,$ from \eqref{1ineq} 
$$\left| \Lambda_1 \right|:=\left| \exp ({\Gamma_1}) -1 \right| <\dfrac{3}{\alpha^{n-k-7}} <\dfrac{1}{2}.$$
Recall that, when $ v <\dfrac{1}{2} ,$ the inequality 
$$ \left| \exp (u) -1 \right| <  v \quad \text{implies that} \quad \left| u  \right| <  2 v.$$
So, it follows that $ \left| \Gamma_1 \right| < \dfrac{6}{\alpha^{n-k-7}}.$ Moreover, $\Gamma_1  \neq 0,$ since $ \Lambda_1 .$ Thus, we write
\begin{equation}
\label{r1} 
0< \left| \dfrac{ d }{n-m }  -  \dfrac{ \log \alpha }{\log 10 } \right | < \dfrac{6}{ (n-m) \alpha^{n-k-7} \log 10}.
\end{equation}
Note that, we have
$$
\dfrac{6}{ (n-m) \alpha^{n-k-7} \log 10} < \dfrac{1}{2(n-m)^2},
$$
for otherwise
$$
\dfrac{6}{ (n-m) \alpha^{n-k-7} \log 10} \geq \dfrac{1}{2(n-m)^2}
$$
implies that
$$
n>n-m> \dfrac{ \alpha^{n-k-7} \log 10  }{ 12 } > \dfrac{ \alpha^{141} \log 10 }{ 12 } > 5.6 \cdot 10^{28},
$$
which is a contradiction because of \eqref{n2}. Then we have that
$$
 0< \left| \dfrac{ \log \alpha }{\log 10 } - \dfrac{ d }{n-m }  \right | < \dfrac{1}{2(n-m)^2},
$$
which implies that $ \dfrac{d}{n-m} $ is a convergent of continued fractions of $ \dfrac{ \log \alpha }{\log 10 } $, say $ \dfrac{ d }{n-m } = \dfrac{p_i}{q_i} .$ Since $\gcd(p_i, q_i)=1$, we have $ q_i \leq n-m \leq n <7 \cdot 10^{26} .$  Let $[a_1,a_2,a_3,a_4,a_5,\ldots]=[0, 4, 1, 3, 1, \ldots ]$ be the continued fraction expansion of $ \log \alpha / \log 10 $. With the help of Maple, we find that $ i < 54 $  and $ \max\{a_i\} =a_{37}=106 $ for $ i=1,2, \ldots , 54. $ So from the well known property of continued fractions, see for example \cite[Theorem 1.1.(iv)]{Hen}, we get that
$$
\dfrac{1}{108 (n-m)^2} \leq \dfrac{1}{(a_i +2) (n-m)^2} <  \left| \dfrac{ \log \alpha }{\log 10 } -  \dfrac{ d }{n-m }  \right |  <  \dfrac{6}{ (n-m) \alpha^{n-k-7} \log 10} 
$$
which means
$$ n> n-m > \dfrac{ \alpha^{n-k-7} \log 10 }{6 \cdot 108}  > 10^{27}. $$
But this is also a contradiction, since $ n < 7 \cdot 10^{26} .$ Therefore, we conclude that $ m \leq 150. $
\end{proof}
By Lemma \ref{mbound}, from now on, we assume that $ m \leq k $ and hence $ n<2k+8 .$  Moreover, we write also $ n-k<m+8 \leq 158.$ By substituting this upper bound for $ n-k $ into the \eqref{2k}, we get a better bound for $ n $ as
\begin{equation}
\label{2k2}
n-8 < 2k < 2 \cdot 2.24 \cdot 10^{12} \cdot (1+\log n) ( 476 \log(\alpha) + 2 ) .
\end{equation}
So from \eqref{2k2}, it follows that
\begin{equation*}
n< 4.5 \cdot 10^{16} . 
\end{equation*}
Let
\begin{equation*} 
\label{gama2}
 \Gamma_2 :=d \log {10} -n \log \alpha + \log  \left( { \dfrac{ F_m  \sqrt{5}}{ 1- \alpha^{k-n} \sqrt{5} } } \right) 
\end{equation*}
so that $ \left | \Lambda_2 \right |:=\left | \exp{(\Gamma_2)}-1 \right | <\dfrac{7}{\alpha^{2k}}.$ Then $ \left | \Gamma_2 \right |  < \dfrac{14} {\alpha^{2k}},$ since $ \dfrac{1}{\alpha^{2k}} < \dfrac{1}{2}.$
So, from \eqref{gama2},

\begin{equation*}
0 < \left | d \dfrac{\log {10}}{\log \alpha} -n +  \dfrac{  \log  \left( { \dfrac{ F_m  \sqrt{5}}{ 1- \alpha^{k-n} \sqrt{5} } } \right)  }{\log \alpha}    \right | < \dfrac{14}{\alpha^{2k} \log \alpha }.
\end{equation*}
Now, we take
$$  
M:=4.5 \cdot 10^{16} > n > d  \qquad \text{and} \qquad \tau:=\dfrac{ \log{ 10}}{\log \alpha}.
$$
Then, in the continued fraction expansion of irrational $ \tau$, we take $ q_{60} ,$ the denominator of the $ 60th $ convergent of  $ \tau,$ which exceeds $ 6M. $  Now, with the help of Maple, we calculate
$$ \epsilon_{m, n-k} := ||\mu_{m, n-k} q_{60} || -M || \tau q_{60} || $$
for each $ 1 \leq m \leq 150 $ and $ 4 \leq n-k <m+8, $   where 
$$ \mu_{m, n-k} :=    \dfrac{  \log  \left( { \dfrac{ F_m  \sqrt{5}}{ 1- \alpha^{k-n} \sqrt{5} } } \right)  }{\log \alpha}, \qquad q_{60}=2568762252997982327345614176552   $$ 
and we see that 
$$ 0.00034 < \epsilon_{50, 20} \leq \epsilon_{m, n-k},  \qquad \text{for all} \quad m, n-k . $$
Let $ A:= \dfrac{14}{\log \alpha} ,$ $ B:=\alpha $ and $ \omega :=2k $ . Thus from  Lemma \ref{reduction}, we find that
$$  2k < \dfrac{ \log{ \left( Aq_{60} /{0.00034} \right) }}{\log B} < 170 .$$
So we get that $ k < 85 $ which is a contradiction because of bound of $ k .$ Thus, we conclude that the numbers 1, 2, 3, 13, 21 and 34 are the only Fibonacci numbers which are expressible as a concatenation of a Fibonacci and a Lucas number.

\textbf{Proof of Theorem \ref{main2}:}
Assume that the equation \eqref{FcLF} holds. As $ d=\lfloor \log_{10}F_k \rfloor +1  $ is the number of digits of $ F_k ,$ we have that
$$  d=\lfloor \log_{10}F_k \rfloor +1 \leq  1+\log_{10}F_k  \leq 1+ \log_{10}{ \alpha^{k-1} }= 1+ {(k-1)} \log_{10}{ \alpha }< \frac{k+3}{4} $$
and
$$  d=\lfloor \log_{10}F_k \rfloor +1 > \log_{10}F_k  \geq \log_{10}{ \alpha^{k-2} } = {(k-2)} \log_{10}{ \alpha }> \frac{k-2}{5} $$
since $ (1/5)< \log_{10}{ \alpha } \cong 0.208... <(1/4) .$ Hence, we have that
\begin{equation}
\label{bd2}
\frac{k-2}{5} < d < \frac{k+3}{4},
\end{equation}
and
$$ F_k  \leq 10^{\lfloor \log_{10}F_k \rfloor} <10^d < 10^{1+\lfloor \log_{10}F_k \rfloor}  \leq 10 F_k .$$
From the last inequality together with \eqref{FcLF}, we can  find a range of $ n $ depending on $ m $ and $ k. $  More precisely,
$$\alpha^{n-1} \leq F_n=L_m 10^d +F_k \leq L_m 10 F_k+F_k <11 L_m F_k < 22 \alpha^{m+k-1} < \alpha^{m+k+7}  $$
and
$$ \alpha^{n-2} \geq F_n=L_m 10^d +F_k \geq L_m F_k+F_k >  L_m F_k > \alpha^{m+k-3}.$$
Hence
\begin{equation}
\label{bn2}
m+k-1 < n < m+k+7.
\end{equation}
When $ k,m \leq 200 $, by a quick computation, we see that the only solutions of \eqref{FcLF} are those given in Theorem \ref{main2}. So from now on, we assume that $ \max\{m,k \} >200. $ Using Binet formula for Fibonacci and Lucas sequences, we rewrite equation \eqref{FcLF} as
$$  \dfrac{\alpha^{n}}{\sqrt{5}}  - \dfrac{\beta^{n}}{\sqrt{5}}    =  \left(  {\alpha^{m} + \beta^{m}}  \right) 10^d +F_k,$$
$$  \dfrac{\alpha^{n}}{\sqrt{5}}  -  \alpha^{m} 10^d  =  \dfrac{\beta^{n}}{\sqrt{5}} +  \beta^{m} 10^d +F_k,$$
Now, we multiply both sides of the above equation by $ {\sqrt{5}}/{\alpha^{n}}  ,$ and take the absolute value of both sides of it. Thus, we get
\begin{align*}
\left| 1-  \dfrac{10^{d} }{ \alpha^{n-m} / \sqrt{5} }  \right|  &  \leq \dfrac{\left| \beta \right |^{n}}{\alpha^n} +   \dfrac{\left| \beta \right |^{m}10^d \sqrt{5} }{\alpha^n} + \dfrac{F_k \sqrt{5} }{\alpha^n} \\
& < \dfrac{  1}{ \alpha^{2n}} +   \dfrac{ 10 F_k \sqrt{5}  }{ \alpha^{n+m}}+ \dfrac{\alpha^{k-1} \alpha^{2}}{\alpha^n} \\
& < \dfrac{  1}{ \alpha^{2n}} +   \dfrac{ \alpha^{5} \alpha^{k-1} \alpha^{2}  }{ \alpha^{n+m}}+ \dfrac{1}{\alpha^{n-k-1}} \\
& < \dfrac{ 3}{ \alpha^{n-k-6}}.
\end{align*}
So, we have that
\begin{equation}
\label{3ineq}  
 \Lambda_3 := \left| 1-  \dfrac{10^{d} \sqrt{5} }{ \alpha^{n-m}  }  \right|  < \dfrac{ 3}{ \alpha^{n-k-6}}.
\end{equation}
Now, we may apply Theorem \ref{Matveev} to the left side of the inequality \eqref{3ineq} with
$(\eta_1, b_1)=(10, d)$, $(\eta_2, b_2)=(\alpha, -(n-m))$ and $(\eta_3, b_3)=(\sqrt{5} , 1)$.
Since $ \eta_1,$ $\eta_2$ and $ \eta_3 $ belong to the real quadratic number field $ \mathbb{K}=\mathbb{Q}(\sqrt{5}) $, we take $ d_\mathbb{K}=2 $, to be the degree of the number field $ \mathbb{K}.$ Since
$  h(\eta_1)=\log{10},$  $h(\eta_2)=\dfrac{1}{2} \log {\alpha},$ and  $h(\eta_3)=\log \sqrt{5}$ we take
$ A_1 = 2\log{10},$  $ A_2= \log {\alpha}$  and  $A_3= \log 5.$
Assume that $ \Lambda_3=0. $ Then, we get that $ \alpha^{n-m} = 10^d \sqrt{5},$ that is,  $ \alpha^{2(n-m)} = 5 \cdot 10^{2d} \in \mathbb{Q} ,$ but this is false for $ n-m \neq 0 $ and one can see that the case $ n-m=0 $ is not possible from  \eqref{FcLF}.  So $ \Lambda_3 \neq 0.$ 

Now, we claim that $ \max\{d, n-m \}=n-m .$ Indeed, if   
$ n-m < d$ then from the inequalities \eqref{bd2} and  \eqref{bn2}, we get that $ k-1<n-m<d<(k+3)/4.$ Thus $ k<3 ,$ which means that $ d=1 $ and hence $ n-m<1, $ that is $ n=m. $  But from the identity $L_i=F_{i-1}+F_{i+1}, $ we see that the case $ n=m $ is not possible. So 
$$
B: = \max\{ |b_i | \} = \max\{ d, n-m, 1 \} = n-m.
$$
Now, we are ready to apply Theorem \ref{Matveev} to $ \Lambda_3 $ given in \eqref{3ineq}, and we get that
$$
\log{( \Lambda_3 ) } > -1.4 \cdot 30^6 \cdot 3^{4.5} \cdot 2^2 (1+\log 2)(1+\log {(n-m)}) \cdot 2 \log{10} \cdot \log \alpha  \cdot  \log 5.
$$
Combining this inequality with the one directly obtained from \eqref{3ineq}, as $ \log( \Lambda_3 )< \log 3 -(n-k-6)\log \alpha  ,$ we get that
\begin{equation}
\label{bn-k2}
n-k-6 < 7.19 \cdot 10^{12} (1+\log {(n-m)}).
\end{equation}
Now, we rewrite \eqref{FcLF} as
$$\dfrac{\alpha^{n}}{\sqrt{5}} - L_m 10^d - \dfrac{\alpha^{k}}{\sqrt{5}} =   \dfrac{\beta^{n}}{\sqrt{5}}  +  \dfrac{\beta^{k}}{\sqrt{5}} ,$$
$$\dfrac{\alpha^{n}}{\sqrt{5}} \left( 1-  \alpha^{k-n} \right) - L_m 10^d =   \dfrac{\beta^{n}}{\sqrt{5}}  +  \dfrac{\beta^{k}}{\sqrt{5}} .$$
After dividing both sides of the last equation by
$ \dfrac{\alpha^{n}}{\sqrt{5}} \left( 1-  \alpha^{k-n} \right) ,$ 
and taking the absolute value of both sides of it, we get that
$$ 
\left| 1- \dfrac{ L_m 10^d \sqrt{5} }{  \alpha^n ( 1-\alpha^{k-n} ) }  \right| \leq      \left(    \dfrac{1}{ 1-\alpha^{k-n} }     \right) \left(  \dfrac{ |\beta^n | }{\alpha^n }  +  \dfrac{ |\beta^k |}{\alpha^n }  \right).
$$
Taking into account $ 1< \dfrac{1}{ 1-\alpha^{k-n} } <2 $ for $ n-k \geq 2,$ we get that
\begin{equation}
\label{4ineq}
\left| \Lambda_4 \right | :=  \left| 1- \dfrac{ L_m 10^d \sqrt{5} }{  \alpha^n \left( 1- \alpha^{k-n} \right) } \right| < \dfrac{4}{\alpha^{2k}}.
\end{equation}
Let
$(\eta_1, b_1)=(10, d)$, $(\eta_2, b_2)=(\alpha, -n)$ and $(\eta_3, b_3)=\left( \dfrac{ L_m \sqrt{5} }{  1- \alpha^{k-n}   }  , 1 \right) .$ All of  $ \eta_1, \eta_2$ and $ \eta_3 $  belong to real quadratic number field $ \mathbb{K}=\mathbb{Q}(\sqrt{5}) $, which has degree $ d_\mathbb{K}=2 .$
Since
$h(\eta_1)= \log{10},$  $h(\eta_2)=(1/2)\log(\alpha)$
and 
$$
h(\eta_3) \leq  h( L_m ) +  h( \sqrt{5} )  +   h( 1- \alpha^{k-n} ) \leq  \log(L_m) + \log(\sqrt{5})  + h( \alpha^{k-n})  +\log 2
$$
$$
<  \log (2\alpha^m) + \log(\sqrt{5})  + \dfrac{|k-n|}{2} \log \alpha  +\log 2.
$$
So 
$$
h(\eta_3) < \dfrac{3(n-k)+2}{2} \log(\alpha) + \log(4\sqrt{5})
$$
where we used the fact that $ m<n-k+1 $ from \eqref{bn2}.
So we take $ A_1 = 2\log {10},$   $A_2= \log \alpha $ and $A_3= (3(n-k)+2) \log(\alpha) + \log(80).$ Clearly, $ B:=\max\{ |b_i | \} = \max\{ d, n, 1 \}= n.$ Also $ \Lambda_4 \neq 0.$ To show this fact assume that $ \Lambda_4=0. $ Then, we get that $\alpha^{n} - \alpha^{k} =   10^d L_m \sqrt{5}. $ Conjugating in  $ \mathbb{K} ,$ we find $\beta^{n} - \beta^{k} =  - 10^d L_m \sqrt{5}. $ Adding side by side the last two equalities we obtain that $ L_n-L_k=0, $ that is $ n=k, $ a contradiction. So $ \Lambda_4 \neq 0.$ Thus, we apply Theorem \ref{Matveev}  to $ \Lambda_4 $ given in \eqref{4ineq}, and we get that
\begin{equation}
\label{b232}
\log{(  \Lambda_4|) } > -1.4 \cdot  30^6 \cdot 3^{4.5} \cdot 2^2 (1+\log 2)(1+\log n) A_3 2\log{10} \log \alpha.
\end{equation}
On the other hand, from \eqref{4ineq}, taking into account that $\log{(  \Lambda_4 ) } < \log4 -2k \log \alpha ,$ we get that
\begin{equation}
\label{b2k}
k < 2.24 \cdot 10^{12} \cdot (1+\log n) [ (3(n-k)+2) \log(\alpha) + \log(80) ] .
\end{equation}
Now, we use the two inequalities  \eqref{bn-k2} and \eqref{b2k} to get an initial bound on the variable $n$. To do this, first assume that $ k \leq m. $ Then, $ n-m \leq n-k $ and hence, from \eqref{bn-k2}, we write 
$$
n-k-6 < 7.19 \cdot 10^{12} (1+\log {(n-k)})
$$
which implies that $ n-k < 3 \cdot 10^{14}.$ So it follows that 
\begin{equation*}
n< 10^{15},
\end{equation*}
since from \eqref{bn2}, we know that $ n<2m+4$ and $ m<n-k+1.$

Let  $ m \leq k .$ From \eqref{bn-k2}, in particular, we have that
\begin{equation*}
\label{bn-k23}
n-k-6 < 7.19 \cdot 10^{12}(1+\log {n}).
\end{equation*}
Substituting \eqref{bn-k23} into \eqref{b2k} and using the fact that $ n<2k+7 ,$ we find that
\begin{equation}
\label{bn23}
n< 2.5 \cdot 10^{29}.
\end{equation}
So, whether $ m \leq k $ or not, we have that $ n< 2.5 \cdot 10^{29}. $

\begin{lemma} \label{bmbound}
If the equation \eqref{FcLF} holds then $ m \leq 168. $  In particular, if $ k \leq m ,$ then the equation has the solution only for  $ F_n \in \{ 13, 21 \} .$
\end{lemma}
\begin{proof}
Suppose that $ m>168.$ Let 
\begin{equation}
\label{gama3} 
\Gamma_3 :=d \log {10} -(n-m) \log \alpha + \log \sqrt{5} .
\end{equation}
Since  $ 161<m-7< n-k-6,$  $ \left| \Lambda_3 \right|:=\left| \exp {(\Gamma_3)} -1 \right| <\dfrac{3}{\alpha^{n-k-6}} < \dfrac{1}{2}.$
Hence, we get that $ \left| \Gamma_3 \right| < \dfrac{6}{\alpha^{n-k-6}}.$
Thus from \eqref{gama3}, we write
\begin{equation*}
0< \left| d \dfrac{ \log {10} }{\log \alpha }  - (n-m) + \dfrac{\log \sqrt{5}}{\log \alpha} \right | < \dfrac{6}{ \alpha^{n-k-6} \log \alpha}.
\end{equation*}
Let $ M:= 2.5 \cdot 10^{29} > n > d$ and  $\tau = \dfrac{\log {10}}{\log \alpha}.$ Then, in the continued fraction expansion of irrational $ \tau$, we see that $ q_{60},$ the denominator of the $ 60th $ convergent of $ \tau,$ exceeds $ 6M.$ With the help of Maple, we calculate
$$ 
\epsilon := ||\mu q_{60} || -M || \tau q_{60} ||
$$
where $\mu = \dfrac{\log \sqrt{5}}{\log \alpha}$ and we find that $0.017775 < \epsilon .$ Let $ A:= \dfrac{6}{\log \alpha} ,$ $ B:=\alpha $ and $ \omega :=n-k-6 .$ From  Lemma \ref{reduction}, we find that
$$  n-k-6 < \dfrac{ \log{ \left( Aq_{60}/ \epsilon \right) }  }{\log B} < 160 .$$
But this contradicts the fact that $ 161<m-7< n-k-6.$ So, we conclude that $ m \leq 168 .$
\end{proof}
By Lemma \ref{bmbound}, from now on, we deal with only the case $ m<k. $  Since $ n-k<m+7  \leq 175 ,$ by substituting this upper bound for $ n-k $ into \eqref{b2k}, we get that
\begin{equation}
\label{b2k2}
n <2k +7 < 2 \cdot 2.24 \cdot 10^{12} \cdot (1+\log n) [ 527 \log(\alpha) + \log(80) ] +7 .
\end{equation}
So from \eqref{b2k2}, it follows that
\begin{equation*}
n< 4.6 \cdot 10^{16} . 
\end{equation*}
Let
\begin{equation}
\label{gama4} 
\Gamma_4 :=d \log {10} -n \log \alpha + \log { \left( \dfrac{ L_m \sqrt{5}}{  1- \alpha^{k-n}}   \right) }. 
\end{equation}
So
$$ 
\left | \Lambda_4 \right |:=\left | \exp{\Gamma_4}-1 \right | <\dfrac{4}{\alpha^{2k}}.
$$
Then, $\left | \Gamma_4 \right |  < \dfrac{8}{\alpha^{2k}},
$ since $ \dfrac{1}{\alpha^{2k}} < \dfrac{1}{2}.$
Thus, from \eqref{gama4},
\begin{equation}
\label{br2}
0 < \left | \dfrac{\Gamma_4}{\log \alpha} \right | < \dfrac{}{\alpha^{2k} \log \alpha }.
\end{equation}
Now, we take $ M:=4.6 \cdot 10^{16} > n > d $ and $\tau:=\dfrac{ \log{ 10}}{\log \alpha}$ which is irrational.
Then, in the continued fraction expansion of $ \tau$, we take $ q_{91} ,$ the denominator of the $ 91th $ convergent of  $ \tau,$ which exceeds $ 6M. $  Now, with the help of Maple we calculate
$$ \epsilon_{m, n-k} := ||\mu_{m, n-k} q_{91} || -M || \tau q_{91} || $$
for each $ 0 \leq m \leq 168 $ and $ 3 \leq n-k <m+7 ,$   where 
$$ \mu_{m, n-k} :=    \dfrac{ \log { \left( \dfrac{ L_m \sqrt{5}}{  1- \alpha^{k-n}}   \right) }   }{\log \alpha}   $$  except for $ n-k=4$ and $ n-k=8 .$

For these two values $ \epsilon_{m, 4}<0  $ and $ \epsilon_{m, 8} <0 .$ To overcome this problem, we use the periodicity of Fibonacci sequences. With an easy computation, one can see that there is no positive integer $ t $  in the range $ 1 \leq t \leq 20 $ which satisfies neither $ F_t \equiv F_{t+4} \pmod 5  $ nor $ F_t \equiv F_{t+8} \pmod 5 .$ Since the period of the Fibonacci sequence modulo 5 is 20 \cite[Theorem 35.6.]{Koshy}, we conclude that there is no positive integer which satisfies either of these two congruences. So, from \eqref{FcLF}, we take $ n-k \neq 4 $ and $ n-k \neq 8. $

Let $ A:= \dfrac{8}{\log \alpha} ,$ $ B:=\alpha $ and $ \omega :=2k .$ Thus, from  Lemma \ref{reduction} we get that, the inequality \eqref{br2} has solutions only for
$$  2k \leq \dfrac{ \log{ \left( Aq_{91} /{0.000109} \right) }}{\log B} \leq 233 .$$
So, we get that $m< k < 117, $ which is a contradiction because of bound of $ k .$ Thus we conclude that the equation \eqref{FcLF} has no solution when $ F_n \not\in \{ 13, 21 \} .$ This completes the proof.

    
\textbf{Discussion :} In this paper, we aimed to search the results of mixed concatenations of the numbers belonging to the different sequences in the Fibonacci and Lucas particular case and we found that there are only a handful of Fibonacci numbers that can be written as a mixed concatenation  of a Fibonacci and a Lucas numbers. The results of this paper together with those of \cite{Alan, Banks} bring the question to the mind that could there exist a nondegenerate \emph{binary} recurrence sequence which contains infinitely many terms which are expressible as a mixed concatenation of, say, a Fibonacci and a Lucas numbers? Let us define a sequence $ u_n = 10F_n+1 .$ Then, clearly, all terms of this sequence are a concatenation of a Fibonacci and $ 1. $ So, in the above question, we can not omit the expression "nondegenerate \emph{binary}". As a final remark, it is also worth noting that just using some properties and identities of Fibonacci and Lucas numbers, we could get only a few partial results so we used the method in this paper, which is the effective combination of Baker's method together with the reduction method.

\end{document}